\documentclass[10pt]{amsart}

\usepackage{amsmath,xspace,amssymb,mathrsfs}
\usepackage{color}

\input xy
\xyoption{all}
\xyoption{2cell}
\UseAllTwocells
\CompileMatrices

\newcommand{\Spec}{\operatorname{Spec}}
\renewcommand{\phi}{\varphi}
\newcommand{\rs}{\operatorname{r.s}}

\newcommand{\Ker}{\operatorname{Ker}}

\newcommand{\Ima}{\operatorname{Im}}

\newcommand{\Ann}{\operatorname{Ann}}

\newtheorem{proposition}{Proposition}[section]
\newtheorem{lemma}[proposition]{Lemma}

\newtheorem{corollary}[proposition]{Corollary}
\newtheorem{theorem}[proposition]{Theorem}

\theoremstyle{definition}
\newtheorem{definition}[proposition]{Definition}
\newtheorem{example}[proposition]{Example}

\newtheorem{remark}[proposition]{Remark}

\usepackage{etoolbox}
\makeatletter
\patchcmd{\@settitle}{\uppercasenonmath\@title}{}{}{}
\patchcmd{\@setauthors}{\MakeUppercase}{}{}{}
\makeatother

\begin{document}

\title[Topological groups and rings]{Some notes on topological rings and their groups of units}

\author[A. Tarizadeh]{Abolfazl Tarizadeh}
\address{Department of Mathematics, Faculty of Basic Sciences, University of Maragheh, Maragheh, East Azerbaijan Province, Iran.}
\email{ebulfez1978@gmail.com}

\date{}

\begin{abstract} If $R$ is a topological ring then  $R^{\ast}$, the group of units of $R$, with the subspace topology is not necessarily a topological group. This leads us to the following natural definition: By an \emph{absolute topological ring} we mean a topological ring such that its group of units with the subspace topology is a topological group.
We prove that every commutative ring with the $I$-adic topology is an absolute topological ring. Next, we prove that if $I$ is an ideal of a ring $R$ then for the $I$-adic topology over $R$ we have $\pi_{0}(R)=R/(\bigcap\limits_{n\geqslant1}I^{n})=t(R)$ where $\pi_{0}(R)$ is the space of connected components of $R$ and $t(R)$ is the space of irreducible closed subsets of $R$. 
We observed that the main result of Koh \cite{kwangil} as well as its corrected form \cite[Chap II, \S12, Theorem 12.1]{Ursul} are not true, and then we corrected this result in the right way. In the Wikipedia pages, it is claimed that ``the identity component of a topological group is always a characteristic subgroup'', we also provide a counterexample to this claim. Finally, we fix a gap in the proof of the fact that every epimorphism of the category of Hausdorff topological spaces has a dense image. 
\end{abstract}

\maketitle

\section{Introduction}

The group of units of a given topological ring with the subspace topology is not necessarily a topological group. This leads us to the notion of an \emph{absolute topological ring} (see Definition \ref{Definition I}). We prove that every commutative ring with the $I$-adic topology is an absolute topological ring (see Theorem \ref{Theorem 3}). Next in Theorem \ref{Theorem I}, we show that the group of units $R^{\ast}$ of a given topological ring $R$ by the topology $\mathscr{T}_{f}$ induced by the map $f:R^{\ast}\rightarrow R\times R$ which is given by $a\mapsto(a,a^{-1})$ is a topological group. This theorem gives us a characterization result (see Corollary \ref{Coro 3 iii}) which asserts that $R^{\ast}$ with the subspace topology $\mathscr{T}$ is a topological group if and only if $\mathscr{T}=\mathscr{T}_{f}$. \\

If $G$ is a topological group then it is shown that every monomial function $G^{n}\rightarrow G$ given by $(x_{1},\ldots,x_{n})\mapsto ax_{1}^{d_{1}}\ldots x_{n}^{d_{n}}$ is continuous where $a\in G$ and each $d_{k}\in\mathbb{Z}$. Similarly, if $R$ is a topological ring then we show that every polynomial function $R^{n}\rightarrow R$ is continuous. This observation has several consequences (especially it unifies various known results as particular cases). \\

In this article, we also give a special consideration to the $I$-adic topology. Especially by using the theory of topological groups and rings, we obtain the following theorem which is one of the main results of this article. First recall that for a given topological space $X$, by $\pi_{0}(X)$ we mean the space of connected components of $X$ and by $t(X)$ we mean the space of irreducible closed subsets of $X$.

\begin{theorem} Let $I$ be an ideal of a commutative ring $R$. Consider the $I$-adic topology over $R$, then we have the following equalities of topological spaces: $$\pi_{0}(R)= R/(\bigcap\limits_{n\geqslant1}I^n)=t(R).$$
\end{theorem}

While trying to understand the proof of the main result of Koh \cite{kwangil} we realized that this result as well as its corrected version \cite[Chap II, \S12, Theorem 12.1]{Ursul} are not true. Then we corrected this result in the right way (see Theorem \ref{Theorem 4 iv}). In Theorem \ref{Theorem 7 seven} we also improve one of the main results of Ganesan \cite[Theorem I]{Ganesan} which asserts that a given nonzero ring is a finite nonfield ring if and only if its zerodivisors is a finite nonzero set. We also correct a mathematical error on Wikipedia web pages (see Remark \ref{Remark 3-three 2025}). 

It is widely known that in the category of Hausdorff topological spaces, epimorphisms are precisely continuous maps with dense images. Although several proofs are presented in the literature to the implication that every epimorphism $f:X\rightarrow Y$ of this category has dense image, there are some gaps in the arguments. For instance, some researchers (see e.g. \cite{Lapuyade}), consider the quotient space $Y/F$ (contracting to a point the closed subspace $F=\overline{f(X)}$ of $Y$) and then take its Hausdorff quotient $H(Y/F)$. Since $f$ is an epic, the canonical map $Y/F\rightarrow H(Y/F)$ is a constant function. So far in this argument, everything is correct. 
But if we assume $F\neq Y$, then no contradiction arises eventually. Hence, this approach does not seem to be true. In the online mathematical literature such as \emph{Mathematics Stack Exchange} and \emph{MathOverflow}, to prove the above implication, some other researchers consider the quotient space $(Y\times\{0,1\})/(F\times\{0,1\})$ where $F=\overline{f(X)}$. But this quotient space is not necessarily Hausdorff. The gap in another proof of the above implication is explained in Remark \ref{Remark 5 five 5}. In this Remark, we also provide a correct proof for this implication.   

In this article, all of the rings are assumed to be commutative. But some of the results (especially Theorems \ref{Theorem I} and \ref{Theorem 4 iv}) can be generalized to noncommutative rings.

\section{Main Results}

If $R$ is a topological ring, then its group of units $R^{\ast}=\{a\in R: \exists b\in R, ab=1\}$ with the subspace topology is not necessarily a topological group. In fact, the group operation of $R^{\ast}$ is the restriction of the multiplication map of $R$ and hence it is continuous. But the inverse map $R^{\ast}\rightarrow R^{\ast}$ which is given by $a\mapsto a^{-1}$ is not necessarily continuous. For example, the adele ring of a global field is a topological ring, but its group of units with the subspace topology is not a topological group (this is well-known and can be found in algebraic number theory books with focusing on adele rings).
This observation leads us to the following notion.

\begin{definition}\label{Definition I} By an \emph{absolute topological ring} (or, \emph{topological ring with continuous inverses}) we mean a topological ring such that its group of units with the subspace topology is a topological group.
\end{definition}

In the following result we will observe that every ring can be made into an absolute topological ring in a canonical and nontrivial way. First recall that if $I$ is an ideal of a ring $R$, then there exists a unique topology over $R$ such that the collection of $a+I^{n}$ with $a\in R$ and $n\geqslant1$ a natural number forms a base for its open subsets. This topology is called the $I$-adic topology.  

\begin{theorem}\label{Theorem 3} Let $I$ be an ideal of a ring $R$.
Then $R$ with the $I$-adic topology is an absolute topological ring.
\end{theorem}

\begin{proof} The additive operation $f:R\times R\rightarrow R$ which is given by $(a,b)\mapsto a+b$ is continuous, because $f^{-1}(a+I^{n})=\bigcup\limits_{r\in R}(r+I^{n})\times(a-r+I^{n})$. The multiplication $g:R\times R\rightarrow R$ which is given by $(a,b)\mapsto ab$ is also continuous, because $g^{-1}(a+I^{n})=\bigcup\limits_{r\in R}V_{r}$ where $V_{r}=(r+I^{n})\times(\bigcup\limits_{\substack{s\in R,\\rs-a\in I^{n}}}s+I^{n})$.
It remains to show that the group of units $R^{\ast}$ with the subspace topology (induced by the $I$-adic topology) is a topological group. Indeed, the inverse map $h:R^{\ast}\rightarrow R^{\ast}$ which is given by $r\mapsto r^{-1}$ is continuous, because $h^{-1}\big((a+I^{n})\cap R^{\ast}\big)=(\bigcup\limits_{\substack{b\in R^{\ast},\\1-ab\in I^{n}}}b+I^{n})\cap R^{\ast}$.
\end{proof}

\begin{remark} In a correspondence with Pierre Deligne, he informed us that the other nice case arises in functional analysis: Every $C^{\ast}$-algebra and more generally every Banach algebra is an absolute topological ring. Also note that using the above definition then a \emph{topological field} means an absolute topological ring such that it is also a field. For example, the field of real numbers with the Euclidean topology is a topological field.
\end{remark}

Recall that if $(X,\mathscr{T})$ is a topological space, $S$ a set and $f:S\rightarrow X$ a map, then clearly the set $\mathscr{T}_{f}=\{f^{-1}(U): U\in\mathscr{T}\}$ is a topology over $S$ and $f$ is made into a continuous map. We call $\mathscr{T}_{f}$ the topology induced by $f$.

\begin{remark}\label{Remark I} Let $f:X\rightarrow Y$ be a continuous map of topological spaces. If $\Ima(f)\subseteq Z\subseteq Y$, then $f$ induces a continuous map $g:X\rightarrow Z$ which is given by $x\mapsto f(x)$ where the topology of $Z$ is the subspace topology. Indeed, $g^{-1}(Z\cap U)=f^{-1}(U)$.
\end{remark}

\begin{lemma}\label{Lemma 1} Let $(R_{k})$ be a family of topological rings. Then the direct product ring $R=\prod\limits_{k}R_{k}$ with the product topology is a topological ring.
\end{lemma}

\begin{proof} It is well-known and easy exercise.
\end{proof}

For a given topological ring $R$, in order to make $R^{\ast}$ a topological group first we extend its topology as follows. Consider the map $f:R^{\ast}\rightarrow R\times R$ given by $a\mapsto(a,a^{-1})$. Clearly the topology over $R^{\ast}$ induced by $f$ is finer than the subspace topology, because $R^{\ast}\cap U=f^{-1}(U\times R)$.

\begin{theorem}\label{Theorem I} Let $R$ be a topological ring and consider the map $f:R^{\ast}\rightarrow R\times R$ which is given by $a\mapsto(a,a^{-1})$. Then $R^{\ast}$ with the topology induced by $f$ is a topological group.
\end{theorem}

\begin{proof} The inverse map $g:R^{\ast}\rightarrow R^{\ast}$ which is given by $a\mapsto a^{-1}$ is continuous, because $g^{-1}\big(f^{-1}(U\times V)\big)=f^{-1}(V\times U)$. Next we show that the group operation $h:R^{\ast}\times R^{\ast}\rightarrow R^{\ast}$ which is given by $(a,b)\mapsto ab$ is continuous. By Lemma \ref{Lemma 1}, the product ring $S:=R\times R$ with the product topology is a topological ring. Hence, its multiplication $g:S\times S\rightarrow S$ which is given by $\big((a,b),(c,d)\big)\mapsto(ac,bd)$
is continuous. Thus the map $\phi:=g\circ(f\times f):R^{\ast}\times R^{\ast}\rightarrow S$ is continuous and we have $Z:=\Ima(\phi)=\Ima(f)$. Then by Remark \ref{Remark I}, $\phi$ induces a continuous map $\psi:R^{\ast}\times R^{\ast}\rightarrow Z$ which is given by $(a,b)\mapsto(ab,a^{-1}b^{-1})$. Then we show that $f$ induces a homeomorphism $\theta:R^{\ast}\rightarrow Z$ onto its image which is given by $a\mapsto f(a)$ where the topology of $Z$ is the subspace topology. Clearly the map $\theta$ is bijective. By Remark \ref{Remark I}, it is  continuous. The map $\theta$ is also an open map, because $\theta\big(f^{-1}(U\times V)\big)=(U\times V)\cap Z$. Hence, $\theta$ is a homeomorphism. Thus its inverse $\theta^{-1}$ and so
$h=\theta^{-1}\circ\psi$ are continuous.
\end{proof}

\begin{corollary}\label{Coro 3 iii} Let $R$ be a topological ring and consider the map $f:R^{\ast}\rightarrow R\times R$ which is given by $a\mapsto(a,a^{-1})$. Then $R^{\ast}$ with the subspace topology $\mathscr{T}$ is a topological group if and only if $\mathscr{T}=\mathscr{T}_{f}$.
\end{corollary}

\begin{proof} If $R^{\ast}$ with the subspace topology $\mathscr{T}$ is a topological group, then the inverse map $g:R^{\ast}\rightarrow R^{\ast}$ is continuous. We have $f^{-1}(U\times V)=R^{\ast}\cap U\cap g^{-1}(R^{\ast}\cap V)$. This shows that $\mathscr{T}_{f}\subseteq\mathscr{T}$. We also have $\mathscr{T}\subseteq\mathscr{T}_{f}$. Hence, $\mathscr{T}=\mathscr{T}_{f}$. The reverse implication follows from Theorem \ref{Theorem I}.
\end{proof}

\begin{remark}\label{Remark II} Remember that if $f,g:X\rightarrow R$ are continuous functions with $X$ a topological space and $R$ a topological ring, then the pointwise addition $f+g:X\rightarrow R$ given by $x\mapsto f(x)+g(x)$ and the pointwise multiplication $f\cdot g:X\rightarrow R$ given by $x\mapsto f(x)g(x)$ are continuous. Indeed, the map $h:X\rightarrow R\times R$ given by $x\mapsto\big(f(x),g(x)\big)$ is continuous, because $h^{-1}(U\times V)=f^{-1}(U)\cap g^{-1}(V)$. Thus $f+g=\alpha\circ h$ and $f\cdot g=\beta\circ h$ are continuous where $\alpha$ and $\beta$ are the addition and multiplication of $R$, respectively. If  $f,g:X\rightarrow G$ are continuous functions with $G$ a topological group, then exactly like the above it can be seen that the pointwise multiplication $f\cdot g:X\rightarrow G$ is continuous.
The set of all continuous functions $X\rightarrow R$ is usually denoted by $C(X,R)$. This set by the above operations is a ring.
It is worth mentioning that the following two special cases of the ring $C(X,R)$ are of particular interest in mathematics (especially in commutative algebra and mathematical analysis) which are including $C(X):=C(X,\mathbb{R})$ and $H_{0}(A):=C(\Spec(A),\mathbb{Z})$ where $A$ is a commutative ring and $\mathbb{Z}$ is equipped with the discrete topology. For the second case see e.g. \cite[Theorem 5.2]{A. Tarizadeh}.
\end{remark}

The above remark leads us to the following result.

\begin{lemma}\label{Lemma 2} $\mathbf{(i)}$ If $G$ is a topological group then every monomial function $G^{n}\rightarrow G$ given by $(x_{1},\ldots,x_{n})\mapsto ax_{1}^{d_{1}}\ldots x_{n}^{d_{n}}$ is continuous where $a\in G$ and each $d_{k}\in\mathbb{Z}$. \\
$\mathbf{(ii)}$ If $R$ is a topological ring then every polynomial function $R^{n}\rightarrow R$ given by $(r_{1},\ldots,r_{n})\mapsto f(r_{1},\ldots,r_{n})$ is continuous where $f(x_{1},\ldots,x_{n})\in R[x_{1},\ldots,x_{n}]$.
\end{lemma}

\begin{proof} (i): For each $k$, the projection map $\pi_{k}:G^{n}\rightarrow G$ given by $(x_{1},\ldots,x_{n})\mapsto x_{k}$ is continuous, because $G^{n}$ is equipped with the product topology. The inverse map $G\rightarrow G$ is also continuous. Hence, the map $G^{n}\rightarrow G$ given by $(x_{1},\ldots,x_{n})\mapsto x_{k}^{-1}$ is continuous. By Remark \ref{Remark II}, the pointwise multiplication $f\cdot g:X=G^{n}\rightarrow G$ of every two continuous functions $f,g:X\rightarrow G$ is continuous. Hence for each $d\in \mathbb{Z}$ the map $(\pi_{k})^{d}:G^{n}\rightarrow G$ given by $(x_{1},\ldots,x_{n})\mapsto x_{k}^{d}$ is continuous. The constant function $h:G^{n}\rightarrow G$ given by $(x_{1},\ldots,x_{n})\mapsto a$ is continuous. Again by Remark \ref{Remark II}, the pointwise multiplication $h\cdot\big(\prod\limits_{k=1}^{n}(\pi_{k})^{d_{k}}\big):G^{n}\rightarrow G$ given by $(x_{1},\ldots,x_{n})\mapsto ax_{1}^{d_{1}}\ldots x_{n}^{d_{n}}$ is continuous. \\
(ii): Similarly to the above case, it can be seen that the monomial function $R^{n}\rightarrow R$ given by $(r_{1},\ldots,r_{n})\mapsto ar_{1}^{d_{1}}\ldots r_{n}^{d_{n}}$ is continuous where $a\in R$ and each $d_{k}\geqslant0$. By Remark \ref{Remark II}, the pointwise addition $g+h:R^{n}\rightarrow R$ of every two continuous functions $g,h:R^{n}\rightarrow R$ is continuous. Thus the map $R^{n}\rightarrow R$ given by $(r_{1},\ldots,r_{n})\mapsto f(r_{1},\ldots,r_{n})$ is continuous.
\end{proof}

\begin{remark} Recall that if $A$ and $B$ are subsets of a group $G$, then $AB=\{ab: a \in A, b\in B\}=\bigcup\limits_{a\in A}aB=\bigcup\limits_{b\in B}Ab$ and $A^{-1}=\{x\in G: x^{-1}\in A\}$. If $U$ is an open subset of a topological group $G$ and $a\in G$ then $U^{-1}$, $aU$ and $Ua$ are open subsets of $G$ and so for any subset $S\subseteq G$, then $SU=\bigcup\limits_{s\in S}sU$ and $US=\bigcup\limits_{s\in S}Us$ are open subsets of $G$. Similarly, if $E\subseteq G$ is a closed subset then $E^{-1}$, $aE$ and $Ea$ are closed subsets. But for an infinite subset $S$, in general neither $ES$ nor $SE$ are closed in $G$.
\end{remark}

\begin{corollary} If $U$ is an open neighborhood of the identity element $e$ of a topological group $G$, then for each $n\geqslant1$ there exists an open neighborhood $V$ of $e$ in $G$ such that $V^{n}\subseteq U$.
\end{corollary}

\begin{proof} By Lemma \ref{Lemma 2}(i), the map $f:G^{n}\mapsto G$ given by $(x_{1},\ldots,x_{n})\mapsto x_{1}\cdots x_{n}$ is continuous, and so $f^{-1}(U)$ is an open subset. Clearly the $n$-tuple $(e,\ldots,e)$ is a member of $f^{-1}(U)$. Thus for each $k$ there exists an open subset $V_{k}$ in $G$ such that $(e,\ldots,e)\in\prod\limits_{k=1}^{n}V_{k}\subseteq f^{-1}(U)$. Then clearly $e\in V:=\bigcap\limits_{k=1}^{n}V_{k}$ and $V^{n}\subseteq U$.
\end{proof}

Recall that for any ring $R$ by $\mathcal{B}(R)=\{e\in R: e=e^{2}\}$ we mean the set of all idempotents of $R$ which is a commutative ring whose addition is $e\oplus e':=e+e'-2ee'$ and whose multiplication is $e\cdot e'=ee'$. We call $\mathcal{B}(R)$ the Boolean ring of $R$. For more information on this ring we refer the interested reader to \cite{Tarizadeh-Taheri}. We know that every subring of a topological ring with the subspace topology is a topological ring. But note that $\mathcal{B}(R)$ is not necessarily a subring of $R$. In spite of this, the property of being a topological ring is still preserved by Booleanization:

\begin{corollary} If $R$ is a topological ring, then the Boolean ring $\mathcal{B}(R)$ with the subspace topology is a topological ring.
\end{corollary}

\begin{proof} The multiplication of $\mathcal{B}(R)$ is the restriction of the multiplication of $R$ and hence it is continuous. Consider the polynomial $f(x,y)=x+y-2xy$ in $R[x,y]$. By Lemma \ref{Lemma 2}(ii), the map $f^{\ast}:R\times R\rightarrow R$ given by $(a,b)\mapsto a+b-2ab$ is continuous.
The addition of $\mathcal{B}(R)$ is the restriction of $f^{\ast}$ and so it is continuous.
\end{proof}

\begin{remark}\label{Remark 3-three 2025} (correcting a mathematical mistake on Wikipedia pages) Let $G$ be a topological group. Then it can be easily seen that the normal subgroup $G_{0}$, the identity component of $G$ (i.e., the connected component of $G$ containing the identity element of $G$), is stable under every continuous automorphism of $G$. But in the Wikipedia pages: $$https://en.wikipedia.org/wiki/Identity_{-}component$$

and also in the following link: 

$$https://en.wikipedia.org/wiki/
Characteristic_{-}subgroup$$

it is claimed that: ``the identity component of a topological group is always a characteristic subgroup." But this is a wrong statement. For example, let $H$ be a connected topological group which is nontrivial (e.g. the additive group of real numbers with the Euclidean topology), and take $H_d$ to be the same group $H$ with the discrete topology. Then the identity component of the topological group $G= H\times H_d$ is $G_{0}= H \times \{e\}$ where $e$ is the identity element of $H_d$. Now consider the (non-continuous) automorphism $f: G\rightarrow G$ defined by $(x,y)\mapsto(y,x)$. Then $f(G_{0})= \{e\} \times H$ which is not contained in $G_{0}$. Hence, $G_{0}$ is not a characteristic subgroup of $G$.   
\end{remark}

\begin{lemma}\label{lemma 5 v} Let $R$ be a topological ring. If $C$ is the connected component of $R$ containing the zero element, then $C$ is an ideal of $R$ and the topological ring $R/C$ is the space of connected components of $R$.
\end{lemma}

\begin{proof} See \cite[Theorem 4.5]{Warner}. 
\end{proof}

\begin{lemma}\label{Lemma 3 iii} Let $I$ be an ideal of a ring $R$. Consider the $I$-adic topology over $R$, then we have: \\
$\mathbf{(i)}$ If $S$ is a subset of $R$, then $\overline{S}=\bigcap\limits_{n\geqslant1}(S+I^{n})$. \\
$\mathbf{(ii)}$ $R$ is a discrete space if and only if $R$ has an isolated point, or equivalently, $I$ is a nilpotent ideal.
\end{lemma}

\begin{proof} It is straightforward.
\end{proof}

\begin{theorem}\label{Coro 1 excellent} Let $I$ be an ideal of a ring $R$. Consider the $I$-adic topology over $R$, then the topological ring $R/(\bigcap\limits_{n\geqslant1}I^{n})$ is the space of connected components of $R$.
\end{theorem}

\begin{proof} Let $J\subseteq R$ be the identity component of $R$ (i.e., the connected component of $R$ containing the zero element). Then $J$ is an ideal of $R$ (see Lemma \ref{lemma 5 v}). We know that every connected component is a closed subset. Then using Lemma \ref{Lemma 3 iii}(i) we have  $J=\overline{J}=\bigcap\limits_{n\geqslant1}(J+I^{n})$. For each $d\geqslant1$, the ideal $I^{d}$ is a base open. It is also a closed subset, because $\overline{I^{d}}=\bigcap\limits_{n\geqslant1}(I^{d}+I^{n})
=I^{d}$ (the closedness of $I^{d}$ also follows from the fact that in a topological group, every open subgroup is closed). We know that in a topological space, a connected subset is contained in a clopen (both open and closed subset) if and only if they meet each other.
Since $0\in J\cap I^{n}$, thus $J\subseteq I^{n}$ and so $J+I^{n}=I^{n}$ for all $n\geqslant 1$. It follows that $J=\bigcap\limits_{n\geqslant1}I^{n}$. Then the assertion follows from Lemma \ref{lemma 5 v}.
\end{proof}

\begin{corollary} Let $I$ be an ideal of a ring $R$. Then $I$ is a connected subset of $R$ with respect to the $I$-adic topology if and only if $I$ is an idempotent ideal. In this case, $R/I$ is the space of connected components of $R$.
\end{corollary}

\begin{proof} If $I$ is connected then it is contained in the connected component of the zero element which is $\bigcap\limits_{n\geqslant1}I^{n}$ by the above result. But $\bigcap\limits_{n\geqslant1}I^{n}\subseteq I^{2}$ and so $I=I^{2}$. Conversely, if $I$ is an idempotent ideal then $I=\bigcap\limits_{n\geqslant1}I^{n}=\overline{\{0\}}$. Thus $I$ is connected.
\end{proof}

The above result, in particular, tells us that if the ideal $I$ is generated by a set of idempotents or more generally it is a pure ideal (i.e., the canonical ring map $R\rightarrow R/I$ is a flat ring map), then $I$ is a connected component of $R$ with respect to the $I$-adic topology. \\

If $I$ is a proper ideal of a ring $R$, then by Theorem \ref{Coro 1 excellent}, $R$ is not connected with respect to the $I$-adic topology.

\begin{corollary} Let $I$ be an ideal of a ring $R$. Consider the $I$-adic topology over $R$, then the following assertions are equivalent. \\
$\mathbf{(i)}$ $R$ is Hausdorff. \\
$\mathbf{(ii)}$ $\bigcap\limits_{n\geqslant1}I^{n}=0$. \\
$\mathbf{(iii)}$ $R$ is totally disconnected. \\
$\mathbf{(iv)}$ $R$ has a connected component which is singleton.
\end{corollary}

\begin{proof} (i)$\Leftrightarrow$(ii): Well known. \\
(ii)$\Rightarrow$(iii): If $\bigcap\limits_{n\geqslant1}I^{n}=0$ then by Theorem \ref{Coro 1 excellent}, $R$ is the space of connected components of $R$. In other words, every connected component of $R$ is singleton. \\
(iii)$\Rightarrow$(iv): There is nothing to prove, because $R$ is nonempty. \\
(iv)$\Rightarrow$(ii): By hypothesis and Theorem \ref{Coro 1 excellent}, we have $\{a\}=a+\bigcap\limits_{n\geqslant1}I^{n}$ for some $a\in R$. It follows that $\bigcap\limits_{n\geqslant1}I^{n}=0$.
\end{proof}

\begin{corollary}\label{Coro 2 ex} Let $I$ be an ideal of a ring $R$ and $x\in R$. Consider the $I$-adic topology over $R$, then $\overline{\{x\}}=x+\bigcap\limits_{n\geqslant1}I^{n}$.
\end{corollary}

\begin{proof} By Theorem \ref{Coro 1 excellent}, the subset $x+\bigcap\limits_{n\geqslant1}I^{n}$ is the connected component of $x$. Hence, it contains the connected subset $\overline{\{x\}}$. To see the reverse inclusion, take $y\in\bigcap\limits_{n\geqslant1}I^{n}$. If $U\subseteq R$ is an open neighborhood of $x+y$ then $x+y\in x+y+I^{d}=x+I^{d}\subseteq U$ for some $d\geqslant1$. It follows that $x\in U$. Hence, $x+y\in\overline{\{x\}}$. Thus $\overline{\{x\}}=x+\bigcap\limits_{n\geqslant1}I^{n}$.
\end{proof}

Recall from \cite[Chap II, \S2, p. 78]{Hartshorne} that if $X$ is a topological space, then by $t(X)$ we mean the set of all irreducible and closed subsets of $X$. It can be easily seen that the set $t(X)$ is a topological space whose closed subsets are precisely of the form $t(E)$ where $E$ is a closed subset of $X$. The canonical map $X\rightarrow t(X)$ given by $x\mapsto\overline{\{x\}}$ is  continuous. If $f:X\rightarrow Y$ is a continuous map of topological spaces, then the map $t(f):t(X)\rightarrow t(Y)$ given by $Z\mapsto\overline{f(Z)}$ is continuous. In fact, $t(-)$ is a covariant functor from the category of topological spaces to itself. In this regard, we have the following result.

\begin{theorem} Let $I$ be an ideal of a ring $R$. Consider the $I$-adic topology over $R$, then the topological space $t(R)$ and the quotient space $R/(\bigcap\limits_{n\geqslant1}I^{n})$ are the same.
\end{theorem}

\begin{proof} If $Z\in t(R)$ then $Z$ is an irreducible and closed subset of $R$. Since $Z$ is nonempty, we may choose some $x\in Z$ and so $\overline{\{x\}}\subseteq Z$. We know that in a topological space, every irreducible subset is connected. So $Z$ is contained in the connected component of $x$. Then using Theorem \ref{Coro 1 excellent} and Corollary \ref{Coro 2 ex}, we have $Z\subseteq x+\bigcap\limits_{n\geqslant1}I^{n}=\overline{\{x\}}$. Therefore, $Z=x+\bigcap\limits_{n\geqslant1}I^{n}$. This shows that the underlying sets are the same, i.e., $t(R)=R/(\bigcap\limits_{n\geqslant1}I^{n})$. Next we show that their topologies are the same. If $\mathscr{C}$ is a closed subset of the quotient space $R/(\bigcap\limits_{n\geqslant1}I^{n})$ then $E:=f^{-1}(\mathscr{C})$ is a closed subset of $R$ where $f:R\rightarrow R/(\bigcap\limits_{n\geqslant1}I^{n})$ is the canonical map. Then clearly $\mathscr{C}=t(E)$. Hence, $\mathscr{C}$ is a closed subset of $t(R)$. To see the reverse inclusion, take a closed subset $t(F)$ in $t(R)$ where $F$ is a closed subset of $R$. We know that if $H$ is a subgroup of a topological group $G$, then the canonical map $\pi$ from $G$ onto the quotient space $G/H$ given by $x\mapsto xH$ is an open map, because for any subset $U\subseteq G$ we have $\pi(U)=UH=\bigcup\limits_{x\in H}Ux$.
Thus $f$ is an open map, and so $f(U)=U+\bigcap\limits_{n\geqslant1}I^{n}$ is an open subset of the quotient space $R/(\bigcap\limits_{n\geqslant1}I^{n})$ where $U=R\setminus F$. But $t(F)=\{x+\bigcap\limits_{n\geqslant1}I^{n}: x\in F\}$ which is the complement of $f(U)$. Hence, $t(F)$ is a closed subset of the quotient space $R/(\bigcap\limits_{n\geqslant1}I^{n})$. This completes the proof.
\end{proof}

\begin{remark} The canonical map $f: X\rightarrow t(X)$ given by $x\mapsto\overline{\{x\}}$ induces a bijection  $W\mapsto f^{-1}(W)$ from
the topology (the set of open subsets) of $t(X)$ to the topology of $X$. We also observe that $f$ is a closed map if and only if every irreducible and closed subset of $X$ has a generic point. Indeed, for any closed subset $E\subseteq X$ we have $f(E)\subseteq t(E)$. If every irreducible closed subset of $X$ has a generic point then $f(E)=t(E)$ and so $f$ is a closed map. Conversely, if $Z$ is an irreducible closed subset of $X$ then $f(Z)=t(E)$ for some closed $E\subseteq X$, it follows that $Z\subseteq E$ thus $Z\in t(E)$ and so $Z=\overline{\{x\}}$ for some $x\in Z$.
Moreover, if $U$ and $V$ are open subsets of $X$ with $f(U)=f(V)$, then $U=V$. The map $E\mapsto t(E)$ is also a bijection from the set of closed subsets of $X$ onto the set of closed subsets of $t(X)$. But in general, $f$ is not an open map. For example, let $X$ be an infinite set equipped with the cofinite topology (i.e., the proper closed subsets of $X$ are the finite subsets). Then $X$ is an irreducible space with no generic point. It is clear that the points of $t(X)$ are precisely $X$ and all of the singletons. Now if $U$ is a nonempty open subset of $X$ then $f(U)$ is not open. Indeed, suppose it is open then $t(X)\setminus f(U)=t(E)$ for some closed subset $E$ of $X$. But $X\in t(E)$ and so $E=X$. It follows that $f(U)$ is the empty set which is a contradiction. Also, for any open subset $U$ in $X$ we have $X\notin f(U)$ and so $f(U)\neq t(X)$.
\end{remark}

By a compact space we mean a quasi-compact and Hausdorff topological space.

\begin{remark} Remember that by a \emph{perfect map} we mean a continuous map $f:X\rightarrow Y$ between topological spaces such that it is a closed map and for each $y\in Y$ the fiber $f^{-1}(y)$ is quasi-compact. For example, every continuous map from a quasi-compact space into a Hausdorff space is a perfect map.
It is well-known and easy to check that the inverse image of every quasi-compact subset under a perfect map is quasi-compact.
\end{remark}

We need the following well-known and fundamental result in the next theorem.

\begin{theorem}\label{Theorem 5 five} For a topological group $G$ the following assertions hold. \\
$\mathbf{(i)}$ If $S$ is a subset of $G$ then $\overline{S}=\bigcap\limits_{U\in\mathcal{N}(e)}SU=
\bigcap\limits_{U\in\mathcal{N}(e)}\overline{SU}$ where $\mathcal{N}(e)$ denotes the set of open neighborhoods of the identity element $e\in G$.\\
$\mathbf{(ii)}$ If $E$ is a closed subset of $G$ and $K$ is a quasi-compact subset of $G$, then $EK$ and $KE$ are closed subsets of $G$. \\
$\mathbf{(iii)}$ If $H$ is a quasi-compact subgroup of $G$ then the canonical map $f:G\rightarrow G/H$ given by $x\mapsto xH$ is a perfect map where the set $G/H$ is equipped with the quotient topology. In this case, $G$ is quasi-compact if and only if $G/H$ is quasi-compact.
\end{theorem}

\begin{proof} (i) and (ii): Well known, see e.g. Karl Hofmann's notes entitled: Introduction to Topological Groups, Lemma 1.15. Note that $E^{-1}$ is a closed subset of $G$ and $K^{-1}$ is a quasi-compact subset of $G$, thus $E^{-1}K^{-1}$  and so $KE=(E^{-1}K^{-1})^{-1}$ are closed subsets of $G$. \\
(iii): It is also well known. Indeed, for any subset $E\subseteq G$ we have $f^{-1}\big(f(E)\big)=EH$. Thus by
(ii), $f$ is a closed map. Each fiber of $f$ is of the form $xH$ which is homeomorphic to $H$ and hence it is quasi-compact.
\end{proof}

Note that the converse of Theorem \ref{Theorem 5 five}(iii) holds trivially: If the canonical map $G\rightarrow G/H$ is a perfect map for some subgroup $H$, then $H$ is quasi-compact.

\begin{remark} Recall that if $f:G\rightarrow H$ is a surjective morphism of topological groups (i.e. a surjective and continuous map of group morphisms), then the induced map $G/N\rightarrow H$ with $N=\Ker(f)$ is an isomorphism (homeomorphism) of topological groups if and only if $f$ is an open map. If $f$ is a closed map then the induced map $G/N\rightarrow H$ is also an isomorphism.
\end{remark}

By $Z(R)=\{a\in R: \Ann(a)\neq0\}$ we mean the set of all zerodivisors of a ring $R$.  \\

The main result of Koh \cite{kwangil} is not true. Indeed, in a given topological ring $R$, the canonical bijective continuous map from the quotient space $R/\Ann(x)$ onto the subspace $Rx$ given by $r+\Ann(x)\mapsto rx$ with $x\in R$ is not necessarily a homeomorphism, even if $Rx$ (or more strongly, every principal ideal of $R$) is a closed subset of $R$. An example can be found in \cite[Chap II, \S12, Remark 12.1]{Ursul}. In the following result, we correct Koh's result in the right way.

\begin{theorem}\label{Theorem 4 iv} Let $R$ be a topological ring which is  Hausdorff and the map $f:R\rightarrow R$ given by $r\mapsto rx $ is a closed map
for some $0\neq x\in Z(R)$. If $Z(R)$ is a compact subset, then $R$ is compact.
\end{theorem}

\begin{proof} The induced map $g:R/\Ker(f)\rightarrow Rx$ given by $r+\Ker(f)\mapsto rx$ is bijective and continuous. It is also a closed map, because $f$ is a closed map. Hence, $g$ is a homeomorphism from the quotient space $R/\Ker(f)$ onto the subspace $Rx$. Since $x\neq0$, so $Rx\subseteq Z(R)$. Clearly $Rx$ is a closed subset of $R$, since $f$ is a closed map.
Thus $Rx$ is quasi-compact, because every closed subset of a quasi-compact space is quasi-compact. Hence, the quotient space $R/\Ker(f)$ is quasi-compact. Since $R$ is Hausdorff, so the zero ideal is a closed point. Thus the fiber $f^{-1}(0)=\Ker(f)$ is also a closed subset of $R$.  Also $\Ker(f)=\Ann(x)\subseteq Z(R)$, since $x\neq0$. Hence, $\Ker(f)$ is quasi-compact. We know that the additive group of every topological ring is a topological group. Thus by Theorem \ref{Theorem 5 five}(iii), $R$ is quasi-compact.
\end{proof}

Note that in the above result, $\Ker(f)$ is a closed subset of $R$ if and only if $R$ is Hausdorff. Because if $\Ker(f)$ is closed then its image under the closed map $f$ is a closed subset which equals to the zero ideal, and so $R$ is Hausdorff (for the reverse implication see the above proof).
Hence, a corrected version of Koh's result \cite[Chap II, \S12, Theorem 12.1]{Ursul} is not true without the ``Hausdorffness" assumption. Also note that in Theorem \ref{Theorem 4 iv}, $f$ is a closed map if and only if the induced map $R/\Ker(f)\rightarrow R$ is a closed map. Indeed, by Theorem \ref{Theorem 5 five}(iii), the canonical map $R\rightarrow R/\Ker(f)$ is a closed map.

\begin{remark}\label{Remark 4 iv} Recall from the basic group theory that if $I$ is an ideal of a ring $R$ such that $I$ and $R/I$ are finite sets then $R$ is a finite ring with $|R|=|I|\cdot|R/I|$.
\end{remark}

The following result improves \cite[Theorem I]{Ganesan}.

\begin{theorem}\label{Theorem 7 seven} Let $R$ be a nonzero ring. Then $R$ is a finite nonfield ring if and only if $Z(R)$ is a finite nonzero set.
\end{theorem}

\begin{proof} The implication ``$\Rightarrow$" is clear, because if $Z(R)=\{0\}$ then $R$ will be an integral domain which is a contradiction since every finite integral domain is a field. Conversely, suppose $Z(R)$ is a finite nonzero set. Consider the discrete topology over $R$. Then by Theorem \ref{Theorem 4 iv}, $R$ is compact and so it is finite. Also $R$ is not a field, because $Z(R)\neq0$. \\
Motivated by the proof of \cite[Theorem I]{Ganesan}, we provide a second proof for the reverse implication without using Theorem \ref{Theorem 4 iv}. Assume $Z(R)$ is a finite nonzero set. So we may choose some $0\neq x\in Z(R)$. Then clearly $I:=\Ann_{R}(x)\subseteq Z(R)$. Hence, $I$ is a finite set. The map $R/I\rightarrow Z(R)$ given by $r+I\mapsto rx$ is injective. Thus $R/I$ is also a finite set. Then by Remark \ref{Remark 4 iv}, $R$ is a finite ring. Moreover,  $|R|=|I|\cdot|R/I|\leqslant n^{2}$ where $n:=|Z(R)|$.
\end{proof}

Note that in the above result, the assumption $Z(R)\neq0$ is vital. For example, the ring of integers $\mathbb{Z}$ has finitely many zerodivisors (the zero element is the only zerodivisor), but it is an infinite ring.

\begin{theorem} Let $f:R\rightarrow R'$ be a morphism of rings, $I$ an ideal of $R$ and $J$ an ideal of $R'$. Then $f$ is continuous with respect to the corresponding $I$-adic and $J$-adic topologies if and only if $f(I^{n})\subseteq J$ for some $n\geqslant1$.
\end{theorem}

\begin{proof} Assume $f$ is continuous. We know that $J$ is an open subset of $R'$ and so $f^{-1}(J)$ is an open subset of $R$. But $0\in f^{-1}(J)$. So there exists some $a\in R$ and a natural number $n\geqslant1$ such that $0\in a+I^{n}\subseteq f^{-1}(J)$. It follows that $a\in I^{n}$ and so $f(I^{n})\subseteq J$. To see the converse, it will be enough to show that $f^{-1}(b+J^{d})$ is an open subset of $R$ where $b\in R'$ and $d\geqslant1$. Take $r\in f^{-1}(b+J^{d})$.
By hypothesis, $f(I^{nd})\subseteq J^{d}$ and so $r\in r+I^{nd}\subseteq f^{-1}(b+J^{d})$. Hence, $f^{-1}(b+J^{d})$ is an open set.
\end{proof}

As an immediate consequence of the above result, if $I$ and $J$ are ideals of a ring $R$ then the $J$-adic topology is contained in the $I$-adic topology (in other words, the $I$-adic topology is finer than the $J$-adic topology) if and only if $I^{n}\subseteq J$ for some $n\geqslant1$. In particular, if $\mathfrak{p}$ and $\mathfrak{q}$ are prime ideals of $R$ then  $\mathfrak{p}$-adic and $\mathfrak{q}$-adic topologies are the same if and only if $\mathfrak{p}=\mathfrak{q}$. \\

Remember that a subset $E$ of a topological space $X$ is called \emph{locally closed} if for each point $x\in E$ there is an open neighborhood $U\subseteq X$ of $x$ such that $U\cap E$ is a closed subset of $U$ (clearly this notion is a generalization of the closed subset). We can  generalize it a little further as: a subset $E$ of a topological space $X$ is called \emph{weak closed} if there exists some open $U\subseteq X$ such that $U\cap E$ is a nonempty closed subset of $U$. This notion enables us to reformulate a well-known technical result in a more simple way:

\begin{theorem} In a topological group, every weak closed subgroup is closed.
\end{theorem}

\begin{proof} See \cite[Chap I, Theorem 4.11]{Warner}.
\end{proof}

\begin{corollary} Every finite weak closed subset of a topological group which is closed under the group operation is a closed subgroup.
\end{corollary}

\begin{proof} It is well-known and easy to check that every finite nonempty subset of a group which is closed under the group operation is a subgroup. Then by the above theorem, it is also a closed subset.
\end{proof}

\begin{example} The ring of integers modulo two $\mathbb{Z}_{2}=\{0,1\}$ with the Sierpi\'{n}sky topology $\mathscr{T}=\{\emptyset,\mathbb{Z}_{2},\{0\}\}$ is not a topological ring, because the additive map $f:\mathbb{Z}_{2}\times\mathbb{Z}_{2}\rightarrow\mathbb{Z}_{2}$ is not continuous: $f^{-1}(0)=\{(0,0),(1,1)\}$ is not open.
\end{example}

The implication ``$\Rightarrow$" of the following result is well-known. We show that the revere is also true:

\begin{lemma}\label{Theorem 10 on} Let $H$ be a subgroup of a topological group $G$. Then  $H$ is dense in $G$ if and only if the quotient topology over the set $G/H$ is trivial. 
\end{lemma}

\begin{proof} First assume that $H$ is dense in $G$.  If $U$ is a nonempty open subset of $G/H$, then $f^{-1}(U)$ is a nonempty open subset of $G$ where $f:G\rightarrow G/H$ is the canonical map. Then $f^{-1}(U)\cap xH\neq\emptyset$ for all $x\in G$, because $xH$ is dense in $G$. Thus $xh\in f^{-1}(U)$ for some $h\in H$. Then $xH=xhH\in U$ and so $U=G/H$. 
This shows that the quotient topology over the set $G/H$ is trivial. Conversely, if $U$ is a nonempty open subset of $G$ then $f(U)=\{uH: u \in H\}$ is a nonempty open subset of $G/H$, because the canonical map $f:G\rightarrow G/H$ is an open map. Thus $\{uH: u \in H\}=G/H$ and so $u\in H$ for some $u \in U$. This shows that $H$ is dense in $G$. 
\end{proof}

\begin{remark}\label{Remark 5 five 5} It can be easily seen that in the category of Hausdorff topological spaces, every morphism (continuous map) with dense image is an epimorphism. The converse is also true. More precisely, if $f:X\rightarrow Y$ is an epimorphism of this category, then its image is dense in $Y$. Although several proofs are presented in the literature to the reverse implication, there are some minor gaps in the arguments. For instance,  in the proof of \cite[Examples A 3.14(iii)]{Hofmann} (on pages 795-796), the relation $R$ is not an equivalence relation, because it is not reflexive: if $y \in Y\setminus X$, then $(y,s)$ does not have relation with itself. In what follows, we present a correct proof for this implication and also fix the above gap. The (right) equivalence relation $R$ over the space $Y\times\{0,1\}$ should be defined as: $(a,i)R(b,j)$ if $a=b\in F=\overline{f(X)}$ or $(a,i)=(b,j)$.
It is clear that the equivalence class $[(a,i)]=\{(a,0), (a,1)\}$ if $a\in F$, otherwise $[(a,i)]=\{(a,i)\}$. 
Then the quotient space $Z=(Y\times\{0,1\})/R$ is Hausdorff. Indeed, take two distinct points $x=[(a,i)]$ and $y=[(b,j)]$ in $Z$. The map $h:Z\rightarrow Y$ given by $[(y,i)]\mapsto y$ is continuous, because $h\circ\pi\circ p$ and $h\circ\pi\circ q$ are the identity maps where $p,q:Y\rightarrow Y\times\{0,1\}$  and $\pi:Y\times\{0,1\}\rightarrow Z$ are the canonical maps. 
If $h$ maps $x$ and $y$ to distinct points of $Y$, then take the inverse images of disjoint neighborhoods, because $Y$ is Hausdorff. If $h(x)=h(y)$ then $a=b\in Y\setminus F$ and $i\neq j$. Thus  
$x\in U:=\pi\big((Y\setminus F)\times\{i\}\big)$ and $y\in V:=\pi\big((Y\setminus F)\times\{j\}\big)$. Clearly $U$ and $V$ are disjoint open subsets of $Z$. Hence, $Z$ is Hausdorff. But $\pi\circ p\circ f=\pi\circ q\circ f$. Since $f$ is an epimorphism, $\pi\circ p=\pi\circ q$. This shows that $Y=F=\overline{f(X)}$.
\end{remark}

 

\begin{thebibliography}{10}
\bibitem{Ganesan}
N. Ganesan, Properties of rings with a finite number of zero divisors, Math. Ann. \textbf{157}(3) (1964) 215-218.
\bibitem{Hofmann}
K.H. Hofmann and S.A. Morris, The Structure of Compact Groups, $4^{th}$ Edition,  Walter de Gruyter GmbH, Berlin/Boston (2020).
\bibitem{Hartshorne}
R. Hartshorne, Algebraic Geometry, Springer-Verlag, New York Inc. (1977).
\bibitem{kwangil}
K. Koh, On the set of zero divisors of a topological ring, Canadian Math. Bull. \textbf{10}(4) (1967) 595-596.
\bibitem{Lapuyade}
J. Lapuyade-Lahorgue, The epimorphisms of the category Haus are exactly the image-dense
morphisms, Master. France. cel-01885564 (2018).
\bibitem{A. Tarizadeh}
A. Tarizadeh and P.K. Sharma, Structural results on lifting, orthogonality and finiteness of idempotents, Rev. R. Acad. Cienc. Exactas F\'{\i}s. Nat. Ser. A Mat. (RACSAM), \textbf{116}(1), 54 (2022).
\bibitem{Tarizadeh-Taheri}
A. Tarizadeh and Z. Taheri, Stone type representations and dualities by power set ring, J. Pure Appl. Algebra \textbf{225}(11) (2021) 106737.
\bibitem{Ursul}
M. Ursul, Topological Rings Satisfying Compactness Conditions, Springer Netherlands (2002).
\bibitem{Warner}
S. Warner, Topological Rings, Elsevier Netherlands (1993). \\
\end{thebibliography}
\end{document}